\documentclass{article}
\pdfoutput=1
\usepackage{amsmath,amsthm,amssymb,fullpage}
\newcommand{\f}[2]{\frac{#1}{#2}}
\newcommand{\abs}[1]{\left\lvert{#1}\right\rvert}
\newcommand{\pq}[1]{\left({#1}\right)}
\newcommand{\C}{\mathbb{C}}

\newcommand{\R}{\mathbb{R}}

\DeclareMathOperator{\im}{im}

\DeclareMathOperator{\End}{End}

\newtheorem{theorem}{Theorem}
\newtheorem{lemma}[theorem]{Lemma}
\newtheorem{proposition}[theorem]{Proposition}

\newtheorem{conjecture}[theorem]{Conjecture}
\theoremstyle{definition}

\numberwithin{theorem}{section}
\title{Limit Points of Commuting Probabilities of Finite Groups}
\author{Thomas Browning}
\date{August 2022}
\begin{document}
	\maketitle
	\begin{abstract}
		The commuting probability of a finite group $G$ is the probability that two randomly chosen elements commute.
		Let $S\subseteq(0,1]$ denote the set of all possible commuting probabilities of finite groups.
		We prove that $S\cup\{0\}$ is closed, which was conjectured by Keith Joseph in 1977.
	\end{abstract}
	\section{Introduction}
	For a finite group $G$, the commuting probability of $G$ is defined as
	\[P(G)=\f{\abs{\{(g,h)\in G\times G:gh=hg\}}}{\abs{G}^2}.\]
	The commuting probability of $G$ also has the formula $P(G)=c(G)/\abs{G}$, where $c(G)$ denotes the number of conjugacy classes of $G$.
	For example, the commuting probability of the dihedral group of order 8 is $P(D_4)=5/8$.
	In fact, $5/8$ is the largest possible commuting probability of a nonabelian group \cite{gustafson1973probability}.
	Keith Joseph studied the set of all possible commuting probabilities
	\[S=\{P(G):G\text{ a finite group}\}\subseteq(0,1],\]
	and observed that the intersection
	\[S\cap\left[\frac{7}{16},1\right]=\left\{\frac{7}{16},\frac{1}{2},\ldots,\frac{1}{2}(1+2^{-2n}),\ldots,\frac{17}{32},\frac{5}{8},1\right\}\]
	seemed to be illustrative of the general behavior of $S$ \cite{joseph1977several} \cite{joseph1970commutativity}.
	Notice that the elements of $S$ approach $\frac{1}{2}$ from above, but not from below, and that the $S$ contains the limit point $\frac{1}{2}$.
	This led Joseph to make the following three conjectures \cite{joseph1970commutativity}.
	\begin{conjecture}[Joseph's Conjectures]
		Let $(x_i)_{i=1}^\infty$ be a sequence of elements of $S$ converging to $\ell>0$.
		Then \emph{(1)} $\ell\in\mathbb{Q}$, \emph{(2)} $x_i\geq\ell$ for all but finitely many $i$, and \emph{(3)} $\ell\in S$ \emph{(}which implies $\ell\in\mathbb{Q}$ since $S\subseteq\mathbb{Q}$\emph{)}.
	\end{conjecture}
	Equivalently, Joseph's second conjecture states that $S$ is well-ordered with respect to the opposite ordering, and Joseph's third conjecture states that $S\cup\{0\}$ is closed.
	
	Rusin proved that Joseph's conjectures hold for sequences converging to $\ell>\frac{11}{32}$ by classifying all finite groups $G$ with $P(G)>\frac{11}{32}$ \cite{rusin1979probability}.\footnote{Some minor errors were corrected by \cite{das2011characterisation}.}
	However, Rusin's approach cannot give any information about $S$ on the interval $(0,\frac{1}{4}]$ since it relies on the estimate $P(G)\leq\frac{1}{4}+\frac{3}{4}\frac{1}{\abs{G'}}$.
	Rusin proves this estimate by considering the number of irreducible characters of degree 1.
	By also considering irreducible characters of degree 2, Hegarty proved that Joseph's first two conjectures hold for sequences converging to $\ell>\frac{2}{9}$, but did not say anything about Joseph's third conjecture \cite{hegarty2013limit}.
	Hegarty's work also revealed a connection between commuting probability and Egyptian fractions.
	Eberhard developed this connection and proved Joseph's first two conjectures \cite{eberhard2015commuting}.
	Eberhard made use of a theorem of Peter Neumann that describes the structure of finite groups $G$ with $P(G)$ bounded away from zero \cite{neumann1989two}.
	In this paper, we will use the theorem of Neumann to prove Joseph's third conjecture.
	We now state our main theorem, which implies all three of Joseph's conjectures.
	\begin{theorem}
	    \label{thm:main}
	    Let $(G_i)_{i=1}^\infty$ be a sequence of finite groups whose commuting probabilities are bounded away from zero.
        Then there exists a finite group $H$ and a subsequence $(G_{n_i})_{i=1}^\infty$ whose commuting probabilities satisfy $P(G_{n_i})\to P(H)$ and $P(G_{n_i})\geq P(H)$.
	\end{theorem}
	In terms of the set $S$, Theorem \ref{thm:main} states that if $(x_i)_{i=1}^\infty$ is a sequence of elements of $S$ bounded away from zero, then there exists an element $\ell\in S$ and a subsequence $(x_{n_i})_{i=1}^\infty$ satisfying $x_{n_i}\to\ell$ and $x_{n_i}\geq\ell$.
	
	We will prove Theorem \ref{thm:main} in Section 4.
	After reducing to a special case, the proof concludes by applying Lemma \ref{lem:equidistribution} from Section 3, which is an equidistribution result for commutators.
	
	Joseph's second conjecture implies that $S$ is well-ordered with respect to the opposite ordering.
	Assuming this conjecture, Joseph also asked for the order type of $S$.
	Eberhard's proof of Joseph's first two conjectures narrowed down the possibilities for the order type of $S$ to either $\omega^\omega$ or $\omega^{\omega^2}$ \cite{eberhard2015commuting}.
	From the proof of Theorem \ref{thm:main}, we are able to determine the order type of $S$.
	\begin{theorem}
	    \label{thm:submain}
	    The order type of $S$ \emph{(}with respect to the opposite ordering\emph{)} is $\omega^\omega$.
	\end{theorem}
	\section{Properties of the Commuting Probability}
	We will need the following properties of the commuting probability.
	\begin{proposition}\label{prop:comm_prob_props}
	The commuting probability satisfies the following properties:
	    \begin{enumerate}
	        \item If $G$ and $H$ are finite groups, then $P(G\times H)=P(G)P(H)$.
	        \item If $N$ is a normal subgroup of a finite group $G$, then $P(G)\geq P(G/N)/\abs{N}$.
	        \item For each positive integer $n$, there exists a finite group $G_n$ with $P(G_n)=\frac{1}{n}$.
	    \end{enumerate}
	\end{proposition}
	\begin{proof}
	    The first property is basic and appears in virtually every prior paper on the commuting probability.
	    A specific reference is Lemma 2(v) in \cite{guralnick2005commuting}, where it is deduced from the corresponding formula for conjugacy classes $c(G\times H)=c(G)c(H)$.
	    We have not seen the second property written down before, but it is also elementary and follows from the corresponding inequality $c(G)\geq c(G/N)$.
	    The third property appears as Corollary 5.3.3 in \cite{castelaz2010commutativity}, where each $G_n$ is constructed as a specific finite group of the form $D_{m_1}\times\cdots\times D_{m_k}$ for odd integers $m_1,\ldots,m_k$.
	    Here $D_m$ denotes the dihedral group of order $2m$ with $P(D_m)=\frac{m+3}{4m}$.
	\end{proof}
	We will also need a statement of Neumann's theorem.
	\begin{theorem}[Theorem 2.4 in \cite{eberhard2015commuting}]
		\label{thm:neumann}
		Let $(G_i)_{i=1}^\infty$ be a sequence of finite groups whose commuting probabilities are bounded away from zero.
		Then there exist normal subgroups $K_i\trianglelefteq G_i$ with $K_i'\leq Z(K_i)$ such that the sequences $(\abs{K_i'})_{i=1}^\infty$ and $([G_i:K_i])_{i=1}^\infty$ are bounded.
	\end{theorem}
	Finally, we will need the following lemma regarding the commutator map in groups $K$ with $K'\leq Z(K)$.
	\begin{lemma}
		\label{lem:bimultiplicative}
		Let $K$ be a group with $K'\leq Z(K)$.
		Then the commutator map $K\times K\to K'$ is bimultiplicative \emph{(}multiplicative in each component\emph{)}.
	\end{lemma}
	\begin{proof}
		We will use the convention $[k,l]=klk^{-1}l^{-1}$.
		Then
		\[[k,l_1][k,l_2]=kl_1k^{-1}l_1^{-1}[k,l_2]=kl_1k^{-1}[k,l_2]l_1^{-1}=[k,l_1l_2],\]
		and similarly
		\[[k_1,l][k_2,l]=k_1lk_1^{-1}l^{-1}[k_2,l]=k_1[k_2,l]lk_1^{-1}l^{-1}=[k_1k_2,l].\qedhere\]
	\end{proof}
	\section{An Equidistribution Result}
	In this section, let $(K_i)_{i=1}^\infty$ be a sequence of finite groups with $K_i'\leq Z(K_i)$ whose commutator subgroups $K_i'$ are all isomorphic to each other, and are identified with a fixed group denoted $K'$.
	Note that each subgroup $H\leq K'$ is a normal subgroup of each $K_i$ since $H\leq K_i'\leq Z(K_i)$.
	The center $Z(K_i/H)$ of the quotient $K_i/H$ is given by $Z(K_i/H)=\bar{Z}(K_i/H)/H$, where
	\begin{equation}
	   \label{eq:zbar}
	   \bar{Z}(K_i/H)=\{k\in K_i:[k,l]\in H\text{ for all }l\in K_i\}.
	\end{equation}
	\begin{lemma}
		\label{lem:Zprops}
		Let $H_1,H_2\leq K'$.
		The subgroups $\bar{Z}(K_i/H)\leq K_i$ satisfy the following properties\emph{:}
		\begin{enumerate}
			\item $\bar{Z}(K_i/K')=K_i$.
			\item $\bar{Z}(K_i/(H_1\cap H_2))=\bar{Z}(K_i/H_1)\cap\bar{Z}(K_i/H_2)$.
			\item If $H_1\leq H_2$, then $\bar{Z}(K_i/H_1)\leq\bar{Z}(K_i/H_2)$.
		\end{enumerate}
	\end{lemma}
	\begin{proof}
		These follow directly from (\ref{eq:zbar}).
	\end{proof}
	Before we can state our equidistribution result, we must first construct a specific subgroup $H_0\leq K'$.
	\begin{lemma}
		\label{lem:existsH0}
		There is a smallest subgroup $H_0\leq K'$ with the property that the sequence $([K_i:\bar{Z}(K_i/H_0)])_{i=1}^\infty$ is bounded.
		In other words, for each subgroup $H\leq K'$,
		\begin{align*}
		&\text{the sequence $([K_i:\bar{Z}(K_i/H)])_{i=1}^\infty$ is bounded}&\iff&&H_0\leq H,\\
		&\text{the sequence $([K_i:\bar{Z}(K_i/H)])_{i=1}^\infty$ is unbounded}&\iff&&H_0\not\leq H.
		\end{align*}
	\end{lemma}
	\begin{proof}
		Consider the set
		\[\mathcal F=\{H\leq K':\text{the sequence $([K_i:\bar{Z}(K_i/H)])_{i=1}^\infty$ is bounded}\}.\]
		The lemma will follow from the following properties of $\mathcal F$:
		\begin{enumerate}
			\item $K'\in\mathcal F$,
			\item If $H_1,H_2\in\mathcal F$, then $H_1\cap H_2\in\mathcal F$,
			\item If $H_1\in\mathcal F$ and $H_1\leq H_2\leq K'$, then $H_2\in\mathcal F$.
		\end{enumerate}
		In other words, $\mathcal F$ is a filter in the lattice of subgroups of $K'$.
		The first and third properties of $\mathcal F$ follow from the first and third statements of Lemma \ref{lem:Zprops}.
		For the second property of $\mathcal F$, combining the second statement of Lemma \ref{lem:Zprops} with the inequality $[G:H\cap K]\leq[G:H]\,[G:K]$ gives
		\[[K_i:\bar{Z}(K_i/(H_1\cap H_2))]=[K_i:\bar{Z}(K_i/H_1)\cap\bar{Z}(K_i/H_2)]\leq[K_i:\bar{Z}(K_i/H_1)]\,[K_i:\bar{Z}(K_i/H_2)].\qedhere\]
	\end{proof}
	Before stating our equidistribution result, we first introduce some notation.
	Let $\varphi_i,\psi_i\colon K_i\to K'$ be homomorphisms.
    We will consider the functions $K_i\times K_i\to K'$ given by $(k,\ell)\mapsto\varphi_i(k)[k,\ell]\psi_i(\ell)$.	
	Let $f_i$ denote the function on $K'$ defined by
	\[f_i(a)=\frac{\lvert\{(k,l)\in K_i\times K_i:\varphi_i(k)[k,l]\psi_i(l)=a\}\rvert}{\abs{K_i}^2}.\]
	In other words, $f_i(a)$ is the probability that $\varphi_i(k)[k,\ell]\psi_i(\ell)=a$.
	\begin{lemma}
		\label{lem:equidistribution}
		Assume that for each subgroup $H\leq K'$, the sequence $([K_i:\bar{Z}(K_i/H)])_{i=1}^\infty$ either is bounded or diverges to infinity.
		Then the functions $K_i\times K_i\to K'$ given by $(k,l)\mapsto\varphi_i(k)[k,l]\psi_i(l)$ are equidistributed on the subgroup $H_0$ of Lemma \ref{lem:existsH0}, in the sense that
		\[f_i(a)-\frac{1}{\abs{H_0}}\sum_{b\in H_0}f_i(b)\to0\]
		for each $a\in H_0$.
	\end{lemma}
	We remark that the assumption in Lemma \ref{lem:equidistribution} can always be satisfied by first passing to a subsequence.
	\begin{proof}
	    Let $\widehat{K'}$ denote the set of homomorphisms $K'\to\mathbb{C}^\times$.
		Fourier analysis on the finite abelian group $K'$ gives the decomposition
		\[f_i=\sum_{\chi\in\widehat{K'}}\langle f_i,\chi\rangle\chi,\qquad\text{where}\qquad\langle f_i,\chi\rangle=\f{1}{\abs{K'}}\sum_{a\in K'}f_i(a)\overline{\chi(a)}.\]
		A standard property of characters is that if $G$ is a finite group, and if $\chi\colon G\to\mathbb{C}^\times$ is a homomorphism, then
		\begin{equation}
		    \label{eq:charsum}
		    \sum_{g\in G}\chi(g)=\begin{cases}\abs{G},&\text{if $\chi(g)=1$ for all $g\in G$},\\0,&\text{otherwise}.\end{cases}
		\end{equation}
		This gives the formula
		\[f_i(a)-\frac{1}{\abs{H_0}}\sum_{b\in H_0}f_i(b)=\sum_{\chi\in\widehat{K'}}\langle f_i,\chi\rangle\pq{\chi(a)-\frac{1}{\abs{H_0}}\sum_{b\in H_0}\chi(b)}=\sum_{\substack{\chi\in\widehat{K'}\\H_0\not\leq\ker\chi}}\langle f_i,\chi\rangle\chi(a),\]
		so it suffices to show that $\langle\chi,f_i\rangle\to0$ for each $\chi\in\widehat{K'}$ with $H_0\not\leq\ker\chi$.
		We can compute
		\[\langle\chi,f_i\rangle=\frac{1}{\abs{K'}}\sum_{a\in K'}\chi(a)f_i(a)=\frac{1}{\abs{K'}}\frac{1}{\abs{K_i}^2}\sum_{k\in K_i}\sum_{l\in K_i}\chi(\varphi_i(k)[k,l]\psi_i(l)).\]
		For each $k\in K_i$, the function $l\mapsto\chi([k,l]\psi_i(l))$ is a homomorphism by Lemma \ref{lem:bimultiplicative}.
		If this homomorphism is nontrivial for each $k\in K_i$, then the inner sum vanishes for each $k\in K_i$ by (\ref{eq:charsum}), and there is nothing to prove.
		Otherwise, let $k_i\in K_i$ be such that $\chi([k_i,l]\psi_i(l))=1$ for all $l\in K_i$.
		Similarly, let $l_i\in K_i$ be such that $\chi(\varphi_i(k)[k,l_i])=1$ for all $k\in K_i$.
		Lemma \ref{lem:bimultiplicative} lets us expand
		\[[kk_i^{-1},ll_i^{-1}]=[k,l_i^{-1}][k,l][k_i^{-1},l][k_i^{-1},l_i^{-1}]=[k,l_i]^{-1}[k,l][k_i,l]^{-1}[k_i,l_i].\]
		Then we can compute
		\begin{align}
		\langle\chi,f_i\rangle&=\frac{1}{\abs{K'}}\frac{1}{\abs{K_i}^2}\sum_{k\in K_i}\sum_{l\in K_i}\chi(\varphi_i(k)[k,l]\psi_i(l))\nonumber\\
		&=\frac{1}{\abs{K'}}\frac{1}{\abs{K_i}^2}\sum_{k\in K_i}\sum_{l\in K_i}\chi([k,l_i]^{-1}[k,l][k_i,l]^{-1})\nonumber\\
		&=\frac{1}{\abs{K'}}\frac{1}{\abs{K_i}^2}\sum_{k\in K_i}\sum_{l\in K_i}\chi([kk_i^{-1},ll_i^{-1}][k_i,l_i]^{-1})\nonumber\\
		&=\frac{1}{\abs{K'}}\frac{1}{\abs{K_i}^2}\chi([k_i,l_i])^{-1}\sum_{k\in K_i}\sum_{l\in K_i}\chi([k,l])\label{eq:inner_product}.
		\end{align}
		For each $k\in K_i$, the homomorphism $l\mapsto\chi([k,l])$ is trivial if and only if $k\in\bar{Z}(K_i/\ker\chi)$.
		Then (\ref{eq:charsum}) gives
		\begin{equation}
		    \label{eq:double_sum}
		    \sum_{k\in K_i}\sum_{l\in K_i}\chi([k,l])=\abs{\bar{Z}(K_i/\ker\chi)}\abs{K_i}=\abs{K_i}^2[K_i:\bar{Z}(K_i/\ker\chi)]^{-1}.
		\end{equation}
		Finally, note that $\abs{\chi([k_i,l_i])}=1$ since $K_i$ is a finite group.
		Then taking absolute values of (\ref{eq:inner_product}) and (\ref{eq:double_sum}) gives
		\[\lvert\langle\chi, f_i\rangle\rvert=\frac{1}{\abs{K'}}[K_i:\bar{Z}(K_i/\ker\chi)]^{-1},\]
		which converges to zero by Lemma \ref{lem:existsH0} since $H_0\not\leq\ker\chi$.
	\end{proof}
	\section{Proof of Main Theorem}
	We will deduce Theorems \ref{thm:main} and \ref{thm:submain} from Theorem \ref{thm:key}.
	\begin{theorem}
	    \label{thm:key}
	    Let $(G_i)_{i=1}^\infty$ be a sequence of finite groups whose commuting probabilities are bounded away from zero.
	    Then there exists a subsequence $(G_{n_i})_{i=1}^\infty$ whose commuting probabilities satisfy at least one of the following two properties\emph{:}
	    \begin{enumerate}
	        \item The commuting probabilities $P(G_{n_i})$ are all equal to each other.
	        \item There exists a finite group $H$ and an integer $k\geq2$ such that $P(G_{n_i})\to\frac{1}{k}P(H)$ and $P(G_{n_i})\geq\frac{1}{k}P(H)$.
	    \end{enumerate}
	\end{theorem}
	\begin{proof}[Proof that Theorem \ref{thm:key} implies Theorem \ref{thm:main}]
	In the first case of Theorem \ref{thm:key}, we can set $H_{\text{Thm \ref{thm:main}}}=G_{n_1}$.
	In the second case of Theorem \ref{thm:key}, we can set $H_{\text{Thm \ref{thm:main}}}=H_{\text{Prop \ref{prop:comm_prob_props}.3}}\times H_{\text{Thm \ref{thm:key}}}$, where $H_{\text{Prop \ref{prop:comm_prob_props}.3}}$ is a finite group satisfying $P(H_{\text{Prop \ref{prop:comm_prob_props}.3}})=\frac{1}{k}$ coming from Proposition \ref{prop:comm_prob_props}.3.
	Then Proposition \ref{prop:comm_prob_props}.1 gives
	\[P(H_{\text{Thm \ref{thm:main}}})=P(H_{\text{Prop \ref{prop:comm_prob_props}.3}}\times H_{\text{Thm \ref{thm:key}}})=P(H_{\text{Prop \ref{prop:comm_prob_props}.3}})P(H_{\text{Thm \ref{thm:key}}})=\frac{1}{k}P(H_{\text{Thm \ref{thm:key}}}).\qedhere\]
	\end{proof}
	The remainder of this section will be devoted to proving Theorem \ref{thm:key}.
	By Theorem \ref{thm:neumann}, there exist normal subgroups $K_i\trianglelefteq G_i$ with $K_i'\leq Z(K_i)$ such that the sequences $(\abs{K_i'})_{i=1}^\infty$ and $([G_i:K_i])_{i=1}^\infty$ are bounded.
	By passing to a subsequence, we may assume that the commutator subgroups $K_i'$ are all isomorphic to each other, and that the quotients $G_i/K_i$ are all isomorphic to each other.
	We will identify the commutator subgroups $K_i'$ with a fixed group denoted $K'$, and the quotients $G_i/K_i$ with a fixed group denoted $G/K$.
	Summing over pairs of elements $C,D\in G/K$ gives the formula
	\begin{equation}
	\label{eq:1}
	P(G_i)=\frac{1}{\abs{G/K}^2}\sum_{C\in G/K}\sum_{D\in G/K}\frac{\abs{\{(g,h)\in C_i\times D_i:gh=hg\}}}{\abs{K_i}^2},
	\end{equation}
	where $C_i,D_i\in G_i/K_i$ correspond to $C,D\in G/K$.
	If we fix coset representatives $(g_i,h_i)\in C_i\times D_i$, then we can rewrite the corresponding summand of (\ref{eq:1}) as
	\begin{align}
	\frac{\abs{\{(g,h)\in C_i\times D_i:gh=hg\}}}{\abs{K_i}^2}&=\frac{\abs{\{(k,l)\in K_i\times K_i:(g_ik)(h_il)=(h_il)(g_ik)\}}}{\abs{K_i}^2}\nonumber\\
	&=\frac{\abs{\{(k,l)\in K_i\times K_i:g_ih_i(h_i^{-1}kh_ik^{-1})kl=h_ig_i(g_i^{-1}lg_il^{-1})lk\}}}{\abs{K_i}^2}\nonumber\\
	&=\frac{\lvert\{(k,l)\in K_i\times K_i:\varphi_i^{h_i}(k)[k,l]\varphi_i^{g_i}(l)^{-1}=h_i^{-1}g_i^{-1}h_ig_i\}\rvert}{\abs{K_i}^2},\label{eq:2}
	\end{align}
	where $\varphi_i^{g_i}(l)=g_i^{-1}lg_il^{-1}$, $\varphi_i^{h_i}(k)=h_i^{-1}kh_ik^{-1}$, and $[k,l]=klk^{-1}l^{-1}$.
	We will denote the conjugation action by exponentiation, so that we can write $\varphi_i^{g_i}(l)=l^{g_i}l^{-1}$ and $\varphi_i^{h_i}(k)=k^{h_i}k^{-1}$.
	\subsection{Reduction to Trivial Action}
	The conjugation action of $G_i$ on $K_i$ descends to an action of $G/K$ on $K_i/K_i'$.
	Then for $C,D\in G/K$, we obtain endomorphisms $\varphi_i^C,\varphi_i^D\in\End(K_i/K_i')$ defined by $\varphi_i^C(l)=l^Cl^{-1}$ and $\varphi_i^D(k)=k^Dk^{-1}$.
	If there exists a pair $(g_i,h_i)\in C_i\times D_i$ with $g_ih_i=h_ig_i$, then (\ref{eq:2}) gives the estimate
	\begin{equation}
	\label{eq:3}
	\frac{\abs{\{(g,h)\in C_i\times D_i:gh=hg\}}}{\abs{K_i}^2}\leq\frac{\abs{\{(k,l)\in K_i/K_i'\times K_i/K_i':\varphi_i^D(k)=\varphi_i^C(l)\}}}{\abs{K_i/K_i'}^2}=\frac{\lvert\im\varphi_i^C\cap\im\varphi_i^D\rvert}{\lvert\im\varphi_i^C\rvert\,\lvert\im\varphi_i^D\rvert}.
	\end{equation}
	If no such pair $(g_i,h_i)\in C_i\times D_i$ exists, then (\ref{eq:3}) is trivially true.
	
	By passing to a subsequence, we may assume that for each $C\in G/K$, the sequence $(\lvert\im\varphi_i^C\rvert)_{i=1}^\infty$ either is bounded or diverges to infinity.
	Now consider the set
	\[Q=\{C\in G/K:\text{the sequence $(\lvert\im\varphi_i^C\rvert)_{i=1}^\infty$ is bounded}\}\subseteq G/K.\]
	The computation $\varphi_i^1(k)=k^1k^{-1}=1$ shows that $1\in Q$.
	The identity $\varphi_i^{CD}(k)=\varphi_i^D(k^C)\varphi_i^C(k)$ gives the inequality $\lvert\im\varphi_i^{CD}\rvert\leq\lvert\im\varphi_i^C\rvert\,\lvert\im\varphi_i^D\rvert$, which shows that $Q$ is closed under multiplication.
	Thus, $Q$ is a subgroup of $G/K$.
	If we let $\pi_i\colon G_i\to G/K$ denote the quotient map, then we can split the sum in (\ref{eq:1}) as
	\begin{equation}
	    \label{eq:5}
	    P(G_i)=\frac{1}{[G/K:Q]^2}P(\pi_i^{-1}(Q))+\frac{1}{\abs{G/K}^2}\sum_{\substack{(C,D)\in(G/K)^2\\C\notin Q\text{ or }D\notin Q}}\frac{\abs{\{(g,h)\in C_i\times D_i:gh=hg\}}}{\abs{K_i}^2}.
	\end{equation}
	Now observe that (\ref{eq:3}) gives the bound
	\begin{equation}
	    \label{eq:6}
	    0\leq\sum_{\substack{(C,D)\in(G/K)^2\\C\notin Q\text{ or }D\notin Q}}\frac{\abs{\{(g,h)\in C_i\times D_i:gh=hg\}}}{\abs{K_i}^2}\leq\sum_{\substack{(C,D)\in(G/K)^2\\C\notin Q\text{ or }D\notin Q}}\frac{\lvert\im\varphi_i^C\cap\im\varphi_i^D\rvert}{\lvert\im\varphi_i^C\rvert\,\lvert\im\varphi_i^D\rvert}\to0\text{ as }i\to\infty.
	\end{equation}
	\begin{lemma}
	If Theorem \ref{thm:key} is true for the sequence $(\pi_i^{-1}(Q))_{i=1}^\infty$, then Theorem \ref{thm:key} is also true for the sequence $(G_i)_{i=1}^\infty$.
	\end{lemma}
	\begin{proof}
	If $[G/K:Q]=1$, then $\pi_i^{-1}(Q)=G_i$ and the lemma is tautological.
	Now suppose that $[G/K:Q]\geq2$.
	By Theorem \ref{thm:key} for the sequence $(\pi_i^{-1}(Q))_{i=1}^\infty$, there exists a subsequence $(\pi_{n_i}^{-1}(Q))_{i=1}^\infty$, a finite group $H$, and an integer $k\geq1$ such that $P(\pi_{n_i}^{-1}(Q))\to\f{1}{k}P(H)$ and $P(\pi_{n_i}^{-1}(Q))\geq\f{1}{k}P(H)$.
	By (\ref{eq:5}) and (\ref{eq:6}), we have $P(G_{n_i})\to\f{1}{k[G/K:Q]^2}P(H)$ and $P(G_{n_i})\geq\f{1}{k[G/K:Q]^2}P(H)$.
	Since $[G/K:Q]\geq2$, we are in the second case of Theorem \ref{thm:key} for the sequence $(G_i)_{i=1}^\infty$.
	\end{proof}
	By replacing $G_i$ with $\pi_i^{-1}(Q)$, we may assume that for each $C\in G/K$, the sequence $(\lvert\im\varphi_i^C\rvert)_{i=1}^\infty$ is bounded.
	Then the inequality
	\[\big[K_i/K_i':\bigcap_{C\in G/K}\ker\varphi_i^C\big]\leq\prod_{C\in G/K}[K_i/K_i':\ker\varphi_i^C]=\prod_{C\in G/K}\lvert\im\varphi_i^C\rvert\]
	shows that the subgroups $\bigcap\ker\varphi_i^C\leq K_i/K_i'$ have bounded index in $K_i/K_i'$.
	If we write $\bigcap\ker\varphi_i^C=L_i/K_i'$, then the subgroups $L_i\leq K_i$ have bounded index in $G_i$.
	If we set $N_i=\bigcap_{g\in G_i}L_i^g$ (i.e., the normal core of $L_i$ in $G_i$), then the subgroups $N_i\leq K_i$ have bounded index in $G_i$ and are normal in $G_i$.\footnote{Actually, $L_i$ is already a normal subgroup of $G_i$, but it is easier to just pass to the normal core anyway.}
	By passing to a subsequence, we may identify the commutator subgroups $N_i'\leq K_i'$ with a fixed subgroup $N'\leq K'$, and the quotients $G_i/N_i$ with a fixed group denoted $G/N$.
	This is the same setup as we had at the start of the proof.
	If $N'<K'$, then we are done by strong induction on $\abs{K'}$.
	Otherwise, replacing $K_i$ with $N_i$ allows us to assume that $G_i$ acts trivially on $K_i/K_i'$.
	\subsection{Applying Equidistribution}
	By passing to a subsequence, we may assume (as required for Lemma \ref{lem:equidistribution}) that for each subgroup $H\leq K'$, the sequence $([K_i:\bar{Z}(K_i/H)])_{i=1}^\infty$ either is bounded or diverges to infinity.
	Let $H_0\leq K'$ be the subgroup whose existence is assured by Lemma \ref{lem:existsH0}.
	Then the subgroups $L_i=\bar{Z}(K_i/H_0)\leq K_i$ have bounded index in $G_i$.
	If we set $N_i=\bigcap_{g\in G_i}L_i^g$ (i.e., the normal core of $L_i$ in $G_i$), then the subgroups $N_i\leq K_i$ have bounded index in $G_i$ and are normal in $G_i$.
	Also, $N_i'\leq H_0$ since $N_i\leq\bar{Z}(K_i/H_0)$.
	By passing to a subsequence, we may identify the commutator subgroups $N_i'\leq H_0\leq K'$ with a fixed subgroup $N'\leq H_0\leq K'$, and the quotients $G_i/N_i$ with a fixed group denoted $G/N$.
	This is the same setup as we had at the start of the proof.
	If $N'<K'$, then we are done by strong induction on $\abs{K'}$.
	Otherwise, we have $H_0=K'$.
	
	Returning to (\ref{eq:2}), note that $\varphi_i^{h_i}(k)=k^{h_i}k^{-1}\in K_i'\leq Z(K_i)$ since $G_i$ acts trivially on $K_i/K_i'$.
	Then
	\[\varphi_i^{h_i}(k_1k_2)=(k_1k_2)^{h_i}(k_1k_2)^{-1}=k_1^{h_i}k_2^{h_i}k_2^{-1}k_1^{-1}=k_1^{h_i}\varphi_i^{h_i}(k_2)k_1^{-1}=k_1^{h_i}k_1^{-1}\varphi_i^{h_i}(k_2)=\varphi_i^{h_i}(k_1)\varphi_i^{h_i}(k_2),\]
	which shows that the functions $\varphi_i^{g_i},\varphi_i^{h_i}\colon K_i\to K'$ are homomorphisms.
	Now we can invoke Lemma \ref{lem:equidistribution} to conclude that the functions $K_i\times K_i\to K'$ given by $(k,l)\mapsto\varphi_i^{h_i}(k)[k,l]\varphi_i^{g_i}(l)^{-1}$ are equidistributed on $K'$ (regardless of the choices of coset representatives $(g_i,h_i)\in C_i\times D_i$), in the sense that
	\begin{equation}
	\label{eq:4}
	\frac{\lvert\{(k,l)\in K_i\times K_i:\varphi_i^{h_i}(k)[k,l]\varphi_i^{g_i}(l)^{-1}=a\}\rvert}{\abs{K_i}^2}\to\frac{1}{\abs{K'}}
	\end{equation}
	for each $a\in K'$.
	
	By passing to a subsequence, we may assume that for each pair of elements $C,D\in G/K$, the chosen coset representatives $(g_i,h_i)\in C_i\times D_i$ either satisfy $h_i^{-1}g_i^{-1}h_ig_i\in K_i'$ for all $i$ or satisfy $h_i^{-1}g_i^{-1}h_ig_i\notin K_i'$ for all $i$.
	Then (\ref{eq:2}) and (\ref{eq:4}) show that each summand of (\ref{eq:1}) either converges to $\frac{1}{\lvert K'\rvert}$ (if $h_i^{-1}g_i^{-1}h_ig_i\in K_i'$ for all $i$) or is identically zero (if $h_i^{-1}g_i^{-1}h_ig_i\notin K_i'$ for all $i$).
	If we write $\mathbf{1}_{K_i'}$ for the indicator function of $K_i'$, then we have shown that
	\[P(G_i)\to\frac{1}{\abs{K'}}\frac{1}{\abs{G/K}^2}\sum_{C\in G/K}\sum_{D\in G/K}\mathbf{1}_{K_i'}(h_i^{-1}g_i^{-1}h_ig_i),\]
	where the limiting value does not depend on $i$.
	Now compare this with the formula
	\[P(G_i/K_i')=\frac{1}{\abs{G/K}^2}\sum_{C\in G/K}\sum_{D\in G/K}\mathbf{1}_{K_i'}(h_i^{-1}g_i^{-1}h_ig_i),\]
obtained by applying (\ref{eq:1}) and (\ref{eq:2}) to the group $G_i/K_i'$ with central subgroup $K_i/K_i'\trianglelefteq G_i/K_i'$.
	In particular, the commuting probability $P(G_i/K_i')$ does not depend on $i$, and we have
	\[P(G_i)\to\frac{1}{\abs{K'}}P(G_i/K_i').\]
	Furthermore, Proposition \ref{prop:comm_prob_props}.2 gives
	\[P(G_i)\geq\frac{1}{\abs{K'}}P(G_i/K_i').\]
	If $\abs{K'}>1$, then we are in the second case of Theorem \ref{thm:key}.
	If $\abs{K'}=1$, then the commuting probability $P(G_i)$ does not depend on $i$ and we are in the first case of Theorem \ref{thm:key}.
	\section{Order Type}
	Following Section 5 of \cite{eberhard2015commuting}, we will determine the order type of $S$ from the iterated derived sets of $S\cup\{0\}$.
	For a closed subset $X\subseteq\R$, the subset $X'\subseteq X$ of limit points of $X$ (equivalently, non-isolated points of $X$) is called the derived set of $X$.
	The iterated derived sets $X^\alpha$ for ordinals $\alpha$ are defined by
	\begin{align*}
	    X^0&=X,\\
	    X^{\alpha+1}&=(X^\alpha)',\\
	    X^\alpha&=\bigcap_{\beta<\alpha}X^\beta\ \ \text{if $\alpha$ is a limit ordinal.}
	\end{align*}
	The following proposition summarizes the discussion and results of Section 5 of \cite{eberhard2015commuting}.
	\begin{proposition}[Section 5 of \cite{eberhard2015commuting}]
	    \label{prop:cantorbendixson}
	    Let $X=S\cup\{0\}$.
	    The order type of $S$ \emph{(}with respect to the opposite ordering\emph{)} is $\omega^\alpha$, where $\alpha$ is the unique ordinal for which $X^\alpha=\{0\}$.
	    Moreover, either $\alpha=\omega$ or $\alpha=\omega^2$.
	\end{proposition}
	We will use Theorem \ref{thm:key} to compute the iterated derived sets of $X=S\cup\{0\}$.
	We will write $\f1kX$ to denote the pointwise rescaling of $X$ by a factor of $\f1k$.
	\begin{theorem}
	\label{lem:derived}
	Let $X=S\cup\{0\}$.
	\begin{enumerate}
	    \item The derived set $X'$ is
	    \[X'=\bigcup_{k=2}^\infty\f1kX=\bigcup_p\frac{1}{p}X,\]
	    where $p$ runs over all prime numbers.
	    \item The iterated derived sets $X^n$ for nonnegative integers $n$ are
	    \[X^n=\bigcup_{\Omega(k)\,\geq\,n}\f1kX=\bigcup_{\Omega(k)\,=\,n}\frac{1}{k}X,\]
	    where $\Omega(p_1^{a_1}\cdots p_j^{a_j})=a_1+\cdots+a_j$ denotes the prime omega function.
	    \item The iterated derived set $X^\omega$ is $X^\omega=\{0\}$.
	    \item The order type of $S$ \emph{(}with respect to the opposite ordering\emph{)} is $\omega^\omega$.
	\end{enumerate}
	\end{theorem}
	\begin{proof}
	    We will prove the four statements sequentially.
        \begin{enumerate}
	        \item Let $x\in X'$ be a nonzero limit point of $X$.
            Then we are in the second case of Theorem \ref{thm:key}, so there exists a finite group $H$ and an integer $k\geq2$ such that $x=\f1kP(H)\in\f1kX$.
            This proves the containment
            \[X'\subseteq\bigcup_{k=2}^\infty\frac{1}{k}X.\]
            The next containment
            \[\bigcup_{k=2}^\infty\frac{1}{k}X\subseteq\bigcup_p\frac{1}{p}X\]
            follows from Proposition \ref{prop:comm_prob_props} which tells us that $X$ is closed under multiplication and contains $\frac{1}{k}$ for each positive integer $k$.
            The extraspecial $p$-groups show that $\frac{1}{p}\in X'$, so the final containment
            \[\bigcup_p\frac{1}{p}X\subseteq X'\]
            follows from the fact that $X$ is closed under multiplication.
            \item We will proceed by induction on $n$.
            The base case of $n=0$ states that
            \[X=\bigcup_{k=1}^\infty\f1kX=X,\]
            which follows from the fact that $X$ is closed under multiplication and contains $\frac{1}{k}$ for each positive integer $k$.
            Now inductively assume that the second statement is true for some nonnegative integer $n$.
            The inductive assumption and the first statement of the lemma let us calculate
            \[X^{(n+1)}=\left(\bigcup_{\Omega(k)\,\geq\,n}\frac{1}{k}X\right)'=\bigcup_{\Omega(k)\,\geq\,n}\frac{1}{k}X'=\bigcup_{\Omega(k)\,\geq\,n}\frac{1}{k}\bigcup_p\frac{1}{p}X=\bigcup_{\Omega(k)\,\geq\,n+1}\frac{1}{k}X,\]
            and similarly with each ``$\geq$'' replaced by ``$=$''.
            The second equality uses the observation that any limit point of $X^{n+1}$ will lie above all but finitely many of the sets $\frac{1}{k}X$ and so must be a limit point of one of the finitely many remaining sets $\frac{1}{k}X$.
            \item By the second statement of the lemma, the iterated derived sets $X^n$ for nonnegative integers $n$ satisfy $\{0\}\subseteq X^n\subseteq[0,2^{-n}]$.
            Then the iterated derived set $X^\omega=\bigcap_nX^n$ satisfies
            \[\{0\}\subseteq X^\omega\subseteq\bigcap_n{[0,2^{-n}]}=\{0\}.\]
            \item The third statement of the lemma tells us that the ordinal $\alpha$ in Proposition \ref{prop:cantorbendixson} is $\omega$, so the order type of $S$ (with respect to the opposite ordering) is $\omega^\omega$.\qedhere
        \end{enumerate}
    \end{proof}
	\section*{Acknowledgements}
	Many thanks to Sean Eberhard for pointing out to me that the order type of $S$ can be determined by the refinement of Theorem \ref{thm:main} that is now stated in Theorem \ref{thm:key}.
	This work was partially supported by NSF grant DMS-1646385 (RTG grant).
	\bibliographystyle{plain}
	\bibliography{refs}
\end{document}